\documentclass[11pt]{amsart}

\usepackage{eucal}
\usepackage{amssymb,url,amscd,amsfonts}
\usepackage{latexsym}
\usepackage{hyperref}
\usepackage[all]{xy}
\theoremstyle{plain}
\newtheorem{theorem}{Theorem}

\newtheorem{lemma}[theorem]{Lemma}
\newtheorem{proposition}[theorem]{Proposition}

\numberwithin{equation}{section} 
\numberwithin{theorem}{section} 
\newtheorem{remark}{Remark}

\begin{document}

\title
[Boundary Minkowski Problem for Weingarten curvatures]
{The Boundary Value Minkowski Problem for Weingarten curvatures}


\author[F. F. Cruz]{Fl\'avio F. Cruz}
\address{Departamento de Matem\'atica\\ Universidade Regional do Cariri\\ Campus Crajubar \\ 
Juazeiro do Norte, Cear\'a\\ Brazil\\ 63041-141}
\email{flavio.franca@urca.br}


\begin{abstract}
In this paper, we prove the existence of hypersurfaces in the Euclidean space with prescribed boundary and 
whose $k$\textsuperscript{th} Weingarten curvature equals a given function that depends on the normal of the hypersurface.
The proof is based on the solvability of a fully nonlinear elliptic PDE.
The required \textit{a priori} estimates are established under the natural assumptions that the prescribed boundary is strictly convex 
and the prescribed function satisfies a Serrin type condition.

\end{abstract}

\maketitle


\section{Introduction}
\label{section1}

Let $\psi$ be a positive function defined in the unit sphere $\mathbb{S}^n$. 
The celebrated Minkowski problem consists of finding a closed convex hypersurface $M$ in 
$\mathbb{R}^{n+1}$ whose Gauss curvature $K_{M}$ is given by
\[
K_{M}=\psi\circ \eta
\]
where $\eta: M \rightarrow \mathbb{S}^n$ is the Gauss map of $M$. The existence and uniqueness of hypersurfaces
defined by a prescribed curvature function in terms of its Gauss map has attracted much attention 
for more than a hundred years. Although
the Minkowski problem being probably the most notable one, the corresponding problem for other important Weingarten curvature
functions such as, for example, the mean and scalar curvatures, has received considerable attention recently.

In the 1950s, A. D. Alexandrov \cite{A2} and S.-s Chern \cite{Ss1, Ss2} 
already raised questions regarding prescribing the Weingarten curvatures in terms of the Gauss map.  
In this direction, B. Guan and P. Guan \cite{GG} proved that if $\psi\in C^{2}(\mathbb{S}^n), \psi>0,$ is invariant under a group of
isometries of $\mathbb{S}^n$ without fixed points, then there exists a closed, strictly convex hypersurface in $\mathbb{R}^{n+1}$ 
whose $k$-th Weingarten curvature is $\psi.$ On the other hand, the boundary value Minkowski problem was proposed 
by Alexandrov \cite{A1}, p. 342, and also by Pogorelov \cite{Pog}, p. 657. 
V. Oliker \cite{O} studied the existence in this setting for the Gauss curvature while  
O. Schn\"urer \cite{Schn} considered, with a different approach, the extension of the problem to a class of curvature functions 
that include multiples and powers of the Gauss curvature. However, the boundary value problem for other Weingartein curvature functions 
has not yet been addressed.
 
In this paper, we consider the problem of finding compact, strictly convex hypersurface in   $\mathbb{R}^{n+1}$ with 
prescribed boundary and  whose Weingarten curvature is prescribed as a function defined on $\mathbb{S}^n$ in
terms of its Gauss map. Let us first recall the definition of the Weingarten curvatures for hypersurfaces. 
Let $S_k(\lambda_1, \ldots, \lambda_n)$ be the
$k$-th order elementary symmetric function normalised so that $S_k(1, \ldots, 1)=1$. For a smooth hypersurface $M$ in
$\mathbb{R}^{n+1}$, let $\kappa=(\kappa_1, \ldots, \kappa_n)$ denote the principal curvatures of $M$. The $k$-th
Weingarten curvature $W_k$ of $M$ is defined as
\begin{equation*}
W_k=S_k(\kappa_1, \ldots, \kappa_n), \quad k=1,\ldots, n.
\end{equation*}
For $k=1,2$ and $n,$ $W_k$ corresponds to the mean, scalar and Gauss curvature, respectively.
We also say a smooth submanifold
$\Sigma\subset  \mathbb{R}^{n+1}$ is strictly convex if through every point of $\Sigma$ there passes a nonsingular
support hyperplane, i.e., a hyperplane with contact of order one with respect to which $\Sigma$ lies strictly on one side.

Our main result may be stated as follows.

\begin{theorem}
\label{teorema}
Let $\psi\in C^{l,1}(\mathbb{S}^n) (l\geq 1)$ be a positive function and $\Sigma$ a smooth, closed,
embedded, strictly convex,  codimension $2$ submanifold in $\mathbb{R}^{n+1}$.
There exists a positive constant $K_0>0$, depending on $\Sigma$, such that if $0<\psi\leq K_0$, then 
$\Sigma$ bounds a $C^{l+2, \alpha}$ embedded convex hypersurface $M$  satisfying
\begin{equation}
\label{principal}
W_k=\psi\circ \eta
\end{equation}
where $\eta: M \rightarrow \mathbb{S}^n$ is the Gauss map of $M$.
\end{theorem}

In general, uniqueness does not hold for Theorem \ref{teorema}. 
However, under stronger assumptions on $\Sigma$, Al\' ias, de Lira and Malacarne 
\cite{ALM} established some rigidity results for $W_k=\psi$ constant and  $k\geq 2$. 
We point out that it follows from a result of M. Ghomi \cite{Ghomi} that the assumption on $\Sigma$ 
is equivalent to ask the existence of a suitable subsolution to the problem.
A specific value for the constant $K_0$ may be obtained by taking the highest value among the 
minimum of the $k$-th Weingarten curvature of the ovaloids 
(i.e. closed hypersurface of positive Gauss curvature) that contain $\Sigma$.

An outline of the proof of Theorem \ref{teorema} is as follows: we reduce the proof of Theorem \ref{teorema} to the existence
of solutions for a Dirichlet problem associated to a fully nonlinear elliptic equation over a region $\Omega\subset\mathbb{S}^n.$
The existence of solution is proved by applying the method of continuity and a degree 
theory argument once the \textit{a priori} estimates for the solutions have been established. The assumption on the geometry of $\Sigma$
and a result of M. Ghomi \cite{Ghomi}  allow us to obtain a subsolution
satisfying the boundary condition, which is crucial for the establishment of the \textit{a priori} boundary estimates.
We remark that the boundary estimates obtained here without imposing structure conditions on $\psi$
besides positivity and under the restriction of $W_k$ to the positive cone seem to be new.

The article is organized as follows.  In Section \ref{section2} we list some basic formulae
which are needed later and describe an appropriate analytical formulation for the problem.
In Section \ref{section3} we deal with the {\it a priori} estimates 
for prospective solutions.
Finally in Section \ref{section4} we complete the proof of Theorem \ref{teorema}
using the continuity method and a degree theory argument with the aid of the previously established  estimates.


\section{Preliminaries}
\label{section2}

We continue to use the notations introduced in Section 1 and consider a smooth, closed,
embedded, strictly convex,  codimension $2$ submanifold $\Sigma$ in $\mathbb{R}^{n+1}.$
M. Ghomi \cite{Ghomi}  proved that, under these conditions, $\Sigma$ lies in an ovaloid $O$. 
The Serrin type condition we have to impose on $\psi$ is then
\begin{equation}
\label{subsolution1}
0< \psi\leq K_0:= \min_{p\in O} W_k(p).
\end{equation}
In fact, in all the proofs instead of \eqref{subsolution1} the strict inequality is assumed to hold, but since 
the estimates in section \ref{section3}, as well as the other proofs, do not depend on a 
quantitative bound for the difference  $K_0-\psi$ it is not difficult to see that the general case follows by approximation.
\begin{remark}
For a precise value of the constant $K_0$ given in Theorem \ref{teorema} consider the set $X$ consisting of all ovaloids that contains 
$\Sigma$ and denote
\begin{align*}
K_O=\min_{p\in O} W_k(p), \quad O\in X,
\end{align*}
where $W_k(p)$ denotes the $k$-th Weingarten curvature of $O$ at $p$. Thus Theorem \ref{teorema} holds for $K_0=\max_{O\in X}K_O$.
\end{remark}

Let $B$ be the convex body whose boundary is $O=\partial B$. 
It follows from the Jordan-Brouwer's separation theorem that
$\Sigma$ bounds a connected region $M^{\prime}$ 
in $O.$ Let $E\subset O$ be a neighborhood of $M^\prime$ such that $\overline E\neq O$ and $p_0$ an interior point of
$O\setminus \overline E$. For any smooth strictly convex hypersurface $M\subset B$ 
with $\partial M=\Sigma$ we consider the set $O_M=\overline M\cup (O\setminus M^\prime)$,
which is the boundary of a convex body $B_M$. 

Moving $p_0$ in the direction of the inward normal of $O_M$ we can find a point $p_1$ and $\delta>0$ such that 
\begin{equation}
\label{Czerobounds}
B_\delta(p_1)\subset \textrm{int } B_M.
\end{equation}
Hence, every such hypersurface $M$ may be written 
as a graph over a small sphere $\partial B_\delta(p_1).$ In particular, $M^{\prime}$ may be represented as a radial graph 
$\bar X(x)=\bar{\rho}(x) x, \, x\in\Omega$, over a region $\overline\Omega\subset \partial B_\delta(p_1).$ Therefore, in order to
find a solution $M$ of \eqref{principal}, it is sufficient to solve a Dirichlet problem associated to a fully nonlinear 
second order elliptic equation defined over $\Omega$. Notice that \eqref{subsolution1} implies that
$\bar \rho$ is a subsolution to this problem. In the sequel we assume, w.l.o.g., that 
$p_1$ is the origin and $\delta=1$, i.e. $\Omega\subset \mathbb{S}^n$. 

Let $M$ be a smooth radial graph given by $X(x)=\rho(x)x,$ where $\rho$ is a smooth
function defined in a domain $\Omega\subset\mathbb{S}^n$ and $e_1, \ldots, e_n$ be a smooth
local frame field on $\mathbb{S}^n$.  Let $\sigma_{ij}=\langle e_i, e_j \rangle$ denote the metric on $\mathbb{S}^n$
and let $\sigma^{ij}$ denote its inverse.
Setting $u=1/\rho$, then the metric, the unit normal and second fundamental form of $M$ are given, respectively, by
\begin{align}
\begin{split}
\label{metric}
& g_{ij}=\frac{1}{u^2}(\sigma_{ij}+\frac{1}{u^2}\nabla_iu\nabla_j u) \\
& \eta=-\frac{1}{(u^2+|\nabla u|^2)^{1/2}}(\nabla u-u x)\\
& h_{ij}=\frac{1}{u\sqrt{u^2+|\nabla u|^2}}(u\sigma_{ij}+\nabla_{ij}u)
\end{split}
\end{align}
where $\nabla$ denotes the covariant differentiation on $\mathbb{S}^n,$  
$\nabla\rho=\textrm{grad} \rho$ is the gradient of $\rho$ and $\nabla_{ij}=\nabla_i\nabla_j.$ 
It is well known that (e.g., see \cite{CNSIV})
the principal curvatures of $M$ are the
eigenvalues of the symmetric matrix $A[u]=[a_{ij}]=[\gamma^{ik}h_{kl}\gamma^{jl}],$ 
where $[\gamma^{ij}]$ and its inverse matrix $[\gamma_{ij}]$ are given, respectively, by
\begin{equation}
\label{raiz-gij}
\gamma^{ij}= u\sigma_{ij}-u\frac{\nabla_i u\nabla_j u}{w(u+w)}
\end{equation}
and
\begin{equation}
\label{raiz-gij2}
\gamma_{ij}= \frac{1}{u}\sigma_{ij}+\frac{\nabla_i u\nabla_j u}{u^2(u+w)}.
\end{equation}

Let $\mathcal{K}$ be the set of $n\times n$ positive definite symmetric matrices and
\begin{equation*}
\Gamma^+=\{\lambda=(\lambda_1, \ldots, \lambda_n)\in\mathbb{R}^n \, : \, \lambda_i>0\}.
\end{equation*}
For each $A\in\mathcal{K}$, let $\lambda(A)=(\lambda_1, \cdots, \lambda_n)$ denote the eigenvalues of $A.$ 
We define
\begin{equation}
\label{defincaoF}
F(A)=W_k^{1/k}\big(\lambda(A)\big),\quad A\in \mathcal{K}.
\end{equation}
Consequently, if the radial graph $M$ is a convex solution of \eqref{principal} then the function $u$ defined above is such that
$u=\underline u=1/\bar{\rho}$ on $\partial\Omega$ and satisfies the following
partial differential equation
\begin{equation}
\label{equation}
G(\nabla^2 u,\nabla u, u)=\Psi(\nabla u, u, x), \quad  x\in\Omega\subset\mathbb{S}^n,
\end{equation}
where
\begin{align*}
G(\nabla^2 u,\nabla u, u) = F\big( A[u]\big) \quad\text{and}\quad
\Psi(\nabla u, u, x) =\psi^{1/k}(\eta).
\end{align*}

Hence, we call a positive function $u\in C^2(\Omega)$ 
{\it admissible} if $u\sigma+\nabla^2u$ is positive definite, where $\sigma$ denotes the standard metric of $\mathbb{S}^n$. 
If $u$ is an admissible solution of \eqref{equation}
and $u=1/\bar\rho$ on $\partial\Omega$,
we can recover a strictly convex hypersurface $M$ that solves \eqref{principal} by $X(x)=\left(1/u(x)\right)x$, 
$x\in\Omega$. Therefore, solving problem \eqref{principal} is equivalent to finding an admissible solution of
the Dirichlet problem
 \begin{align}
\begin{split}
\label{equationDP}
 G(\nabla^2 u,\nabla u, u)&=\Psi(\nabla u, u, x) \quad\textrm{in }\, \Omega\\
u&=\varphi \quad\textrm{on }\, \partial\Omega
\end{split}
\end{align}
where $\varphi=1/\bar{\rho}_{ |_{\partial\Omega}}$. By \eqref{subsolution1}, the function $\underline{u}=1/\bar\rho$ is 
a subsolution of equation \eqref{equationDP}, i.e.,
\begin{align}
\begin{split}
\label{subsolution2}
\underline{\psi}(x):= G(\nabla^2 \underline  u,\nabla \underline u, \underline  u) &\geq  \Psi(\nabla \underline u, \underline u, x) 
\quad\textrm{in }\, \Omega\\
\underline u&=\varphi \quad\textrm{on }\, \partial\Omega.
\end{split}
\end{align}

We now proceed to derive a priori estimates for admissible solutions of \eqref{equationDP}.
As we shall work on two auxiliary forms of equation \eqref{equationDP} in sections \ref{section3} 
and \ref{section4}, we are going to represent \eqref{equationDP} generically by
\begin{align}
\begin{split}
\label{equation4}
G(\nabla^2 u,\nabla u, u)&=\Upsilon (\nabla u, u, x) \quad\textrm{in }\, \Omega\\
u&=\varphi \quad\textrm{on }\, \partial\Omega
\end{split}
\end{align}
for a function $\Upsilon$ defined in terms of $\psi$ (see equations \eqref{psi-t} and \eqref{phi-t} below).


\section{A priori estimates}
\label{section3}

In this section we derive the {\it a priori} estimate
\begin{equation}
\label{C2alphaestimate}
\| u\|_{C^{2,\alpha}(\Omega}\leq C
\end{equation}
for admissible solutions $u$  of \eqref{equationDP} satisfying $u\geq \underline{u}.$ 
Notice that the $C^0$ bounds follows from the geometric setting \eqref{Czerobounds} and the convexity of $M$.
In order to derive the $C^1$ bounds on the boundary, we observe that any admissible solution $u$ satisfies
$\Delta_s u+nu>0$, where $\Delta_s$ is the Laplace-Beltrami operator on $\mathbb{S}^n$.
Let $\overline u$ be the solution of
 \begin{align}
\begin{split}
\label{eqsupersolution}
\Delta_s\overline u+n L&=0 \quad\textrm{in }\, \Omega\\
\overline u&=\varphi \quad\textrm{on }\, \partial\Omega
\end{split}
\end{align}
where $L>0$ is a uniform constant satisfying $|u|\leq L\textrm{ in }\Omega.$ 
So, we have $\underline u\leq u\leq \overline u$ on $\Omega$ and, as the tangencial derivatives of $u$ on $\partial \Omega$
is known, it follows that
\begin{equation}
\label{Cumbordo}
|\nabla u|\leq C\quad \textrm{on } \partial\Omega.
\end{equation}

Now we proceed to the estimate of $|\nabla u|$ in the interior of $\Omega$. Consider the function
\begin{equation*}
w=(u^2+|\nabla u|^2)^{1/2}
\end{equation*}
and let $x_0\in\overline\Omega$ a point where $w$ attains its maximum. If  $x_0\in\partial\Omega$ the estimate
follows from \eqref{Cumbordo}. If $x_0\in\Omega$, we have
\begin{align*}
w\nabla_iw(x_0)=(u\sigma_{ij}+\nabla_{ij}u)\sigma^{jk}\nabla_k u=0
\end{align*}
for all $i=1,\ldots, n$. Since $u\sigma+\nabla^2u$ is positive definite, it follows that $\nabla u(x_0)=0$ and we get
$\max_{\overline\Omega} |\nabla u|\leq  w(x_0)=|u(x_0)|\leq L$. Thus, we have the uniform gradient estimate
\begin{align}
\label{Cuminterior}
|\nabla u| \leq C \quad \textrm{in } \Omega.
\end{align}

Next, we shall establish the second derivatives estimates. Let us assume we have a bound
on the boundary
\begin{equation}
\label{Cdoisbordo}
|\nabla^2u|\leq C \quad \textrm{on } \partial\Omega.
\end{equation}
To establish the second derivatives estimate in $\Omega$, we follow the approach used in 
the proof of Proposition 2.1 in \cite{GG} and we include the computations
here just for the convenience of the reader. First, we observe that as $M$ is strictly convex and 
$O_M=\overline M\cup (O\setminus M^\prime)$ is the boundary of a convex
body $B_M$, the Gauss map $\eta: M\rightarrow \mathbb{S}^n$ is a diffeomorphism from $M$ to $\eta(M)\subset\mathbb{S}^n$. 
Consider the supporting function $u$ (the ambiguous use of $u$ should not cause any difficulties) given by
\begin{align}
\label{funcaosuporte}
u(x)=\langle x, \eta^{-1}(x)\rangle, \quad x\in \eta(M)\subset\mathbb{S}^n.
\end{align}
It is well known (e.g., see \cite{GG}) that the original hypersurface $M$ can be recovered from the support function $u$ and
the eigenvalues of $(\nabla^2 u+ u\sigma)$ at $x\in\eta(M)$ with respect to the standard metric $\sigma$
of $\mathbb{S}^n$ are the inverses of the principal curvatures of $M$ at $\eta^{-1}(x)$. Thus, as
\begin{equation}
S_k(\kappa_1,\ldots,\kappa_n)=[S_{n,n-k}(\kappa_1^{-1}, \ldots,\kappa_n^{-1})]^{-1}
\end{equation}
where $S_{n,k}=S_n/S_k$ for $0\leq k\leq n$, defining $\tilde F(A)=\big (S_{n, n-k}(\lambda(A)\big)^{1/k}$,
we can rewrite \eqref{equation4} by
\begin{equation}
\label{Equationsupport}
\tilde{F}(\nabla^2u+u\sigma)=\tilde\psi\quad \textrm{in } \,\, \eta(M)\subset\mathbb{S}^n
\end{equation}
where $\tilde\psi=\psi^{-1/k}$.  Now, consider the function
\begin{align*}
H=\textrm{trace} (\nabla^2u+u\sigma)=\Delta u+nu
\end{align*}
and assume that $H$ attains its maximum at an interior point $\eta^{-1}(x_0)$  of $M$. Choose an orthonormal
local frame $e_1, \ldots, e_n$ of $\mathbb{S}^n$ about $x_0$ such that $\nabla_{ij}u(x_0)$ is diagonal. We denote
\begin{align*}
v_{ij}=\nabla_{ij}u+u\delta_{ij}\quad \textrm{and}\quad \tilde{F}^{ij}=\frac{\partial \tilde F}{\partial v_{ij}}\big(\{v_{ij}\}\big).
\end{align*}
Differentiating $H$ with respect to the standard metric on $\mathbb{S}^n$, we get
\begin{equation}
\label{contaH1}
\nabla_{ii}H=\Delta (v_{ii})-nv_{ii}+H.
\end{equation}
As $M$ is strictly convex, the matrix $\{v_{ij}\}$ is positive definite and hence so is $\{\tilde F^{ij}\}$. Thus,
since $\{H_{ij}\}$ is negative semidefinite at $x_0$, it follows that
\begin{align}
\label{contaH2}
\tilde F^{ii}H_{ii}=\tilde F^{ii}\Delta(v_{ii})-n\tilde F^{ii}v_{ii}+H \sum_i \tilde F^{ii} \leq 0.
\end{align}
By differentiating equation \eqref{Equationsupport} and using that $\tilde F$ is concave, we obtain
\begin{align}
\label{contaH3}
\tilde F^{ii}\Delta(w_{ii})\geq \Delta \tilde\psi.
\end{align}
So, by using the homogeneity of $\tilde{F}$ and the inequality $\sum_i F^{ii}\geq 1$, we can deduce from \eqref{contaH2} that
\begin{align*}
\Delta \tilde \psi -n\tilde\psi+H\leq 0,
\end{align*}
which implies that $H\leq C$ for a uniform constant $C$. Then, an upper bound for $\nabla^2u$ follows from  
the $C^0$ estimates. To establish a lower bound, we first use the Newton-Maclaurin inequality to get
\begin{align}
\label{NewtonMacinequality}
\begin{split}
S_{n}(\kappa_1^{-1}, \ldots,\kappa_n^{-1}) & =\tilde\psi^k S_{n-k}(\kappa_1^{-1}, \ldots,\kappa_n^{-1})
\\& \geq \tilde\psi^k S_{n}(\kappa_1^{-1}, \ldots,\kappa_n^{-1})^{\frac{n-k}{n}}.
\end{split}
\end{align}
Therefore
\begin{align*}
S_{n}(\kappa_1^{-1}, \ldots,\kappa_n^{-1})\geq c_0\tilde\psi^n
\end{align*}
for some uniform constant $c_0>0$. The lower bound for $\nabla^2 u$, then, follows from the upper bound for the eigenvalues
of $\nabla^2u+u\sigma$ and we obtain the estimate
\begin{equation}
\label{Cdoisinterior}
|\nabla^2u|\leq C \quad \textrm{in } \Omega.
\end{equation}
We notice that, if $H$ attains its maximum at a point $x_0\in\partial\Omega$ then the desired estimate \eqref{Cdoisinterior} 
follows from assumption \eqref{Cdoisbordo}.

Now, we shall to establish the second derivatives estimates on the boundary \eqref{Cdoisbordo}. In what 
follows, we (return to) denote $u=1/\rho$ where $\rho$ is defined by 
the radial parametrization of $M$: $X(x)=\rho(x)x,\, x\in\Omega$.
Let  $x_0\in\partial\Omega,$ an arbitrary fixed point and choose a local orthonormal frame 
field $e_1, \ldots, e_n$  on $\mathbb{S}^n$ around $x_0,$ where $e_n$ is the parallel translation of 
the unit normal vector field on $\partial\Omega.$
From the equality $u=\underline{u}$ on $\partial\Omega$, we get
\begin{equation}
\label{U-2-bordo}
\nabla_{ij}(u-\underline{u})=-\nabla_{n}(u-\underline{u})\Pi_{ij}\quad\textrm{for any } i,j<n,
\end{equation} 
where $\Pi_{ij}=\langle \nabla_{e_{i}} e_j, e_n\rangle$ is the second fundamental form of
$\partial\Omega.$ It follows that
\begin{equation}
\label{EST-TANG}
|\nabla_{ij} u|\leq C, \quad i,j<n,
\end{equation}
for a uniform constant $C$. Now, we turn our attention to the estimate of the mixed tangential-normal and
double normal second derivatives on the boundary. Consider the linearized operator
\begin{align}
\label{Loperator}
L=G^{ij}\nabla_{ij}+\tilde{G}^{i}\nabla_i
\end{align}
where
\begin{align*}
G^{ij}=\frac{\partial G}{\partial\nabla_{ij}u}\quad\textrm{and}\quad \tilde G^i=\frac{\partial G}{\partial\nabla_i u}
+g^{ik}\sigma\left( \nabla \psi(N), \frac{1}{u^2w}e_k-\frac{\nabla_lu}{u^3w}x\right)
\end{align*}
and $g^{ij}$ is the inverse of the metric of $M$ given in \eqref{metric}. Differentiating equation 
\eqref{equation4}, we get
\begin{align}
\label{contaLalpha}
\begin{split}
L\big(\nabla_\alpha u \big) &=  G^{ij}\nabla_{ij\alpha} u+\tilde G^i\nabla_{i\alpha}u+G_u\nabla_\alpha u \\
& =  G^{ij}(\nabla_{\alpha ij}u+\delta_{ij}\nabla_\alpha u-\delta_{\alpha j}\nabla_i u)+\tilde G^i\nabla_{i\alpha}u+G_u\nabla_\alpha u
\\ & =  - G_u \nabla_\alpha u-g^{ij}\sigma_{i\alpha}\sigma\left( \nabla \psi(N), \frac{1}{uw}e_j-\frac{\nabla_ju}{u^2w}x\right)
\\& \quad +\nabla_\alpha u\sum G^{ii}-G^{i\alpha}\nabla_i u
\end{split}
\end{align}
where we denote $G_u=\frac{\partial G}{\partial u}$ and make use of the formula for commuting the order of derivatives on 
$\mathbb{S}^n$. By a direct computation and the previous established estimates, we have
(see, e. g. \cite{Cruz})
\begin{align}
\label{EstimateforG}
\sum |G^i| & \leq C\\
|G_u|  \leq C\big( 1 & +\sum G^{ii}\big)
\end{align}
for a uniforme constante $C$ depending on $\Omega$ and $\underline{u}$. Thus,
\begin{equation}
\label{Lestimate1}
|L\big(\nabla_\alpha(u-\underline u)\big)|\leq C\big(1+\sum G^{ii}\big)
\end{equation}
for a uniform constant $C$. In order the introduce the barrier function, we need first to settle some notation. 
Let $\varrho(x)$ denote the distance from $x\in\Omega$ to $x_0,$ 
$\varrho(x)=\textrm{dist}_{\mathbb{S}^n}(x,x_0),$ and set
\begin{equation*}
\Omega_{\delta}=\{x\in\Omega\, :\, \varrho(x)<\delta\}.
\end{equation*}
Since $\nabla_{ij}\varrho^2(x_0)= 2\delta_{ij}$, by choosing $\delta>0$ sufficiently small, we 
can assume that $\varrho$ is smooth in $\Omega_{\delta},$
\begin{equation}
\label{RHO-2}
\sigma_{ij}\leq \nabla_{ij}\varrho^2\leq 3\sigma_{ij} \quad\textrm{in}\,\, \Omega_\delta,
\end{equation}
and the distance function $d(x)=\textrm{dist}_{\mathbb{S}^n}(x,\partial\Omega)$ to the boundary 
$\partial\Omega$ is smooth in $\Omega_\delta.$ 
A slightly different version of this Lemma could be found in \cite{Cruz}. We include its proof here for the convenience of the reader.

\begin{lemma}
\label{DEF-V}
There exist some uniform positive constants $t, \delta , \varepsilon$ sufficiently small and 
$N$ sufficiently large depending on $\underline u$ and other known data, such that the function
\begin{equation}
\label{V}
v=u-\underline{u} +td-Nd^2
\end{equation}
satisfies
\begin{equation}
\label{Lestimate2}
Lv\leq -(1+\varepsilon G^{ij}\delta_{ij}) \quad \textrm{in }  \Omega_\delta
\end{equation}
and
\begin{equation}
\label{theta-bordo}
v\geq 0 \quad\textrm{on } \partial\Omega_\delta.
\end{equation}
\end{lemma}

\begin{proof}
Let $\varepsilon>0$ be such that
\begin{equation}
\label{convex1}
\underline{u} \sigma+\nabla^2\underline{u}\geq 4\varepsilon \sigma \quad \textrm{in } \,\Omega_{\delta}.
\end{equation}
Since $|\nabla d|=1$ and $-C\sigma\leq \nabla^2d\leq C \sigma,$ in $\Omega_\delta$,
for a constant $C$ depending only on the geometry of $\Omega,$ we have 
\begin{equation*}
G^{ij}\nabla_{ij} d\leq C G^{ij}\delta_{ij}
\end{equation*}
 and
\begin{align*}
\begin{split}
\lambda(\underline{u}\sigma+\nabla^2\underline{u}+N \nabla^2d^2-2\varepsilon \sigma) 
\geq \lambda(\underline{u}\sigma+\nabla^2\underline{u}+2N \nabla d\otimes\nabla d-3\varepsilon \sigma)
\end{split}
\end{align*}
in $\Omega_\delta,$ for any $\delta$ sufficiently small.
Using the concavity of $W_k^{1/k}$ we get
\begin{align*}
 F\Big( & \big\{\frac{1}{uw}\gamma^{ik}(\underline{u}\sigma_{kl}
+\nabla_{kl}\underline{u}+2N\nabla_l d\nabla_k d
-3\varepsilon \sigma_{lk})\gamma^{jl}\big\}\Big)-\Upsilon (\nabla u, u, x)\\
& \leq G^{ij}\Big(\nabla_{ij}\underline{u}+\underline{u}\sigma_{ij} 
+N\nabla_{ij}d^2-2\varepsilon\sigma_{ij}-(u\sigma_{ij}+\nabla_{ij}u)\Big) 
\\ & = G^{ij} \nabla_{ij}(\underline{u}-u+Nd^2)+(\underline{u}-u)G^{ij}\sigma_{ij}
-2\varepsilon G^{ij}\sigma_{ij}.
\end{align*}
Then
\begin{align*}
\begin{split}
G^{ij}&\nabla_{ij}(u-\underline{u}+td-Nd^2)\leq \Upsilon (\nabla u, u, x)-2\varepsilon G^{ij}\sigma_{ij}
+tCG^{ij}\sigma_{ij}\\ &- F\Big( \big\{\frac{1}{uw}\gamma^{ik}
(\underline{u}\sigma_{kl}+\nabla_{kl}\underline{u}+2N\nabla_l d\nabla_k d
-3\varepsilon \sigma_{lk})\gamma^{jl}\big\}\Big) \\
 = & - W_k^{1/k}\left( \lambda\Big(\{\frac{1}{uw}\gamma^{ik}
(\underline{u}\sigma_{kl}+\nabla_{kl}\underline{u}-3\varepsilon \sigma_{kl})\gamma^{jl}
+\frac{2N}{uw}\gamma^{ik}\nabla_l d\nabla_k d\gamma^{jl}\}\Big)\right)\\
&+(tC-2\varepsilon)G^{ij}\sigma_{ij}+\Upsilon (\nabla u, u, x)
\end{split}
\end{align*}
for a uniform constant $C$. By the choice of $\varepsilon$ and the previous established 
$C^0$ estimate, there exists a uniform positive constant $\lambda_0$ satisfying
\begin{equation}
\label{convex4}
\left\{\frac{1}{uw}\gamma^{ik}(\underline{u}\delta_{kl}
+\nabla_{kl}\underline{u}-3\varepsilon \delta_{lk})\gamma^{jl}\right\}\geq \lambda_0 \sigma.
\end{equation}
Then, we can find a uniform positive constant $\mu_0$ such that 
\begin{align*}
P^T\big\{\frac{1}{uw}\gamma^{ik}
(\underline{u}\delta_{kl}+&\nabla_{kl}\underline{u}-3\beta \delta_{kl})\gamma^{jl}
+\frac{2N}{uw}\gamma^{ik}\nabla_l d\nabla_k d\gamma^{jl}\big\}P \\ & \geq  
\textrm{diag}\{\lambda_0,\lambda_0, \ldots, \lambda_0+N\mu_0\},
\end{align*}
where $P$ is an orthogonal matrix that diagonalizes
$\big\{\gamma^{ik}\nabla_l d\nabla_k d\gamma^{jl}\big\}.$
Then, by the ellipticity and concavity of $W_k^{1/k}$ in the positive cone $\Gamma^+$, we get
\begin{align*}
\begin{split}
& W_k^{1/k}\left( \lambda\big(\{\frac{1}{uw}\gamma^{ik}
(\underline{u}\delta_{kl}+\nabla_{kl}\underline{u}-3\beta \delta_{lk})\gamma^{jl}
+\frac{2N}{uw}\gamma^{ik}\nabla_l d\nabla_k d\gamma^{jl}\}\big)\right)\\
&=W_k^{1/k}\left( \lambda\big(P^T\{\frac{1}{uw}\gamma^{ik}(\underline{u}\delta_{kl}+
\nabla_{kl}\underline{u}-3\beta \delta_{lk})\gamma^{jl}
+\frac{2N}{uw}\gamma^{ik}\nabla_l d\nabla_k d\gamma^{jl}\}P\big)\right) \\ &
 \geq W_k^{1/k}(\lambda_0,\lambda_0, \ldots, \lambda_0+N\mu_0).
\end{split}
\end{align*} 
Since 
\begin{equation*}
W_k^{1/k}(\lambda_0,\lambda_0, \ldots, \lambda_0+N\mu_0)
\rightarrow +\infty \quad \textrm{as} \quad N\rightarrow+\infty,
\end{equation*}
it follows that, for $t$ small enough such that $Ct\leq \varepsilon$ and $N$ sufficient large, we have
\begin{equation*}
G^{ij}\nabla_{ij}v\leq -1-\varepsilon G^{ij}\delta_{ij}.
\end{equation*}
Finally, choosing $\delta$ even smaller, such that $\delta N<t,$ we get 
$v\geq 0$ on $\partial \Omega_\delta.$ 
\end{proof}

Now, consider the function $Av+B\varrho^2$ where $A$ and $B$ are large constants to be determined as follow: First we choose $B>0$
large enough to ensure $Av+B\varrho^2\geq \pm\nabla_\alpha (u-\overline{u})$ on $\partial \Omega_\delta$, for any
$1\leq \alpha\leq n-1$.
On the other hand, it follows from \eqref{Lestimate1} and \eqref{Lestimate2} that we can choose $A\gg B\gg 1$ suficiently large such that
\begin{equation}
L\big(Av+B\varrho^2\pm\nabla_\alpha(u-\overline{u})\big)\leq 0 \quad \textrm{in }\, \Omega_\delta
\end{equation}
Thus, by the Maximum Principle, we have $Av+B\varrho^2\geq \pm\nabla_\alpha(u-\overline{u})$ in $\Omega_\delta$.
Since $(Av+B\varrho^2)(x_0)=\nabla_\alpha(u-\overline{u})(x_0)=0$, we get
\begin{equation}
-\nabla_n(Av+B\varrho^2)(x_0)\leq\nabla_\alpha(u-\overline{u})(x_0)\leq \nabla_n(Av+B\varrho^2)(x_0)
\end{equation}
which give us the mixed second derivatives boundary  estimate
\begin{equation}
\label{Mixedbounds}
|\nabla_{nk} u|\leq C
\end{equation}
for any $1\leq k\leq n-1$. 

Now, we consider the pure normal second derivative bound. First, we prove the uniform lower bound
\begin{equation}
\label{DoubleNormal}
M=\min_{x\in\partial\Omega} \min_{\xi\in T_x(\partial\Omega), |\xi|=1} (u+\nabla_{\xi\xi}u)\geq c_0
\end{equation}
for some uniform $c_0>0$, where $T_x(\partial\Omega)$ denotes the tangent space of $\partial\Omega$ at $x$.
Suppose $M$ is achieved at $x_0\in\partial\Omega$ with $\xi\in T_{x_0}(\partial\Omega)$. Consider a local orthonormal frame field
$e_1, \ldots, e_n$ around $x_0$, chosen as above,  and such that $e_1(x_0)=\xi$. Thus
\begin{align}
\label{DoubleNormal1}
\begin{split}
M &=  u(x_0)+\nabla_{11}u(x_0) \\ & = \underline{u}(x_0)+\nabla_{11}\underline{u}(x_0)-\nabla_{n}(u-\underline{u})(x_0)\Pi_{11}(x_0)
\end{split}
\end{align}
where we have used \eqref{U-2-bordo} in the second equality. We can assume that
\begin{align}
\label{DoubleNormal2}
\nabla_{n}(u-\underline{u})(x_0)\Pi_{11}(x_0) >\frac{1}{2}\big(\underline{u}(x_0)+\nabla_{11}\underline{u}(x_0)\big)
\end{align} 
for otherwise we are done because of the strictly local convexity of the radial graph $M^\prime$ of 
the function $1/\underline{u}$. So,
\begin{align*}
\Pi_{11}(x)\geq \frac{1}{2}\Pi_{11}(x_0)> \frac{\underline{u}+\nabla_{11}\underline{u}}{4\nabla_{n}(u-\underline{u})}(x_0)\geq c_1>0
\quad \textrm{in }\, \Omega_\delta
\end{align*}
for uniform $c_1>0$ and $\delta>0$ sufficiently small. Hence, the function
\begin{align*}
\mu=\frac{\underline u+\nabla_{11}\underline u-M}{\Pi_{11}}
\end{align*}
is well defined in $\Omega_\delta$. Now we observe that
\begin{align*}
\underline{u}+\nabla_{11}\underline{u}-\nabla_n(u-\underline{u})\Pi_{11}=u+\nabla_{11} u\geq M.
\end{align*}
Thus, the function 
\begin{align}
\Phi(x)=\mu(x)-\nabla_n(u-\underline{u})(x), \quad x\in\Omega_\delta
\end{align}
satisfies $\Phi \geq 0$ on $\partial\Omega_\delta\cap\partial\Omega$. A direct computation shows that
\begin{align*}
L(\Phi)\leq  L(\nabla_n u)+C\big(1+\sum G^{ii}\big).
\end{align*}
Proceeding as in \eqref{contaLalpha} we get
\begin{align*}
 L(\nabla_n u) & =  - G_u \nabla_n u-g^{in}\sigma\big( \nabla \psi(N), \frac{1}{uw}e_i-\frac{\nabla_i u}{u^2w}x\big)
 \\& \quad  -G^{in}\nabla_i u +\nabla_n u\sum G^{ii}.
\end{align*}
Therefore,
\begin{align}
L(\Phi)\leq  C\big(1+\sum G^{ii}\big)
\end{align}
for a uniform positive constant $C$.  Thus, applying Lemma \ref{DEF-V}, we can proceed as above and find positive
constants $A\gg B\gg 1$ so that
\begin{align}
L\big(Av+B\varrho^2+\Phi \big)\leq 0 \quad \textrm{in }\, \Omega_\delta
\end{align}
and
\begin{align}
Av+B\varrho^2+\Phi \geq 0\quad\textrm{on }\,\partial\Omega_\delta.
\end{align}
Hence, the Maximum Principle and equality $(Av+B\varrho^2+\Phi)(x_0)=0$ imply that
\begin{align*}
\nabla_n\big(Av+B\varrho^2+\Phi\big)(x_0)\geq 0
\end{align*}
and we have the uniform upper bound $\nabla_{nn} u(x_0)\leq C$. Since $u$ is admissible, it follows from the previous estimates
the uniform bound $|\nabla^2 u (x_0)|\leq C$. Then, the principal curvatures of $M$ at 
$X(x_0)$ also have an upper bound. To obtain a uniform
positive lower bound we use again the Newton-Maclaurin inequality  \eqref{NewtonMacinequality} to obtain
\begin{align}
\label{Lowerbound}
S_{n}(\kappa_1^{-1}, \ldots,\kappa_n^{-1})\geq c_0>0
\end{align}
for some uniform constant $c_1>0$. So, as each $\kappa_i$ is bounded from above at $X(x_0)$ by a uniform constant,
\eqref{Lowerbound} gives a uniform positive lower bound for each $\kappa_i$ at $X(x_0)$. Therefore, \eqref{DoubleNormal}
is established. Thus, for every $x\in\partial\Omega$ the eigenvalues of $\{u\delta_{\alpha\beta}+
\nabla_{\alpha\beta}u\}(x)_{\alpha,\beta \leq n-1}$ have a uniform positive lower bound, which finally implies 
a uniform upper bound for $u+\nabla_{nn} u$.
So, \eqref{Cdoisbordo} is established. Now, a uniform positive lower bound for the principal curvatures of $M$ 
follows from \eqref{Lowerbound}.

Then, we have the following result.
\begin{proposition}
\label{proposition}
Let $u\geq \underline{u}$ be an admissible solution of \eqref{equation4}. Then, 
\begin{equation}
\label{S3-2}
\| u\|_{C^{2}(\bar{\Omega})}\leq C \quad \textrm{and} \quad C^{-1}\leq \kappa_i\leq C
\end{equation}
where $C$ is a positive constant depending on $\Omega, \inf_{\bar\Omega}\underline u, \|\underline{u}\|_{C^2(\bar\Omega)},
\|\psi\|_{C^{1,1}(\mathbb{S}^n)},$ the convexity of $M^\prime$ and other known data.
\end{proposition}


\section{Proof of Theorem \ref{teorema}}
\label{section4}

In order to proof  Theorem \ref{teorema} we follow the same approach used in \cite{Cruz}, applying 
the method of continuity and a degree theory argument
with the aid of the {\it a priori} estimates we have alredy established.

First, we need to express \eqref{equation4} in a different form. Setting  $v=-\ln \rho=\ln u,$
the matrix  $A[u]=[a_{ij}]$ can be written in terms of $v$ by
\begin{equation}
\label{g-v}
a_{ij}=\frac{e^v}{w}\big(\sigma_{ij}+\gamma^{ik}\nabla_{kl}v\gamma^{jl}\big)
\end{equation}
where
\begin{equation}
\label{h-v}
w=\sqrt{1+|\nabla v|^2} \quad \textrm{and}\quad \gamma^{ij}=\sigma_{ij}-\frac{\nabla_i v\nabla_j v}{w(1+w)}.
\end{equation}
Consider, for each fixed $t\in [ 0,1],$ the functions $\Theta^t$ and $\Xi^t$ defined in 
$D=\mathbb{R}^{n}\times \mathbb{R}\times\Omega$ by
\begin{equation}
\label{psi-t}
\Theta^t(p,z,x)=e^{2(z-\underline v(x))} \big(t\Psi(p,z,x)+(1-t)\underline{\psi}(x)\big)
\end{equation}
and
\begin{equation}
\label{phi-t}
\Xi^t(p,z,x)= \big(t+(1-t)e^{2(z-\underline{v}(x))}\big)\Psi(p,z,x)
\end{equation}
where $\underline v=-\ln\bar\rho=\ln \underline{u}$ is the subsolution 
and $\underline\psi$ is given in \eqref{subsolution2}.

Choosing $\Upsilon=\Theta^t$ in the generic form \eqref{equation4} of equation \eqref{equationDP}, then
\eqref{equation4} takes the form
\begin{align}
\begin{split}
\label{equation5}
H(\nabla^2 v,\nabla v, v)&=e^{2(v-\underline v)} \big(t\Psi(\nabla v, v, x)+(1-t)
\underline{\psi}(x)\big)\quad\textrm{in }\, \Omega\\
v&=\ln \varphi \quad\textrm{on }\, \partial\Omega.
\end{split}
\end{align}
Notice that $\underline v=-\ln\bar\rho=\ln \underline{u}$
is a strictly subsolution of \eqref{equation5} for $t>0$ and it is a solution for $t=0.$
Moreover, as 
\begin{align*}
\begin{split}
H_v-\Theta^t_v = \frac{\partial}{\partial v}\big( H-\Theta^t\big)
 =F^{ij} a_{ij} -2\Theta^t=-\Theta^t\leq 0
\end{split}
\end{align*}
we can apply the comparison principle to equation \eqref{equation5} 
and conclude that any solution $v^t$ for $t>0$ satisfy $v^t> \underline{v}$ in $\Omega$. 
Hence, Proposition \ref{proposition} can be applied and we get the $C^2$ estimates for any solution $v^t$ of \eqref{equation5}.
Therefore, the H\"older estimates follows from the Evans-Krylov Theorem and 
we can apply the continuity method to conclude that a unique solution $v^0$ of
\eqref{equation5} for $t=1$ exists. Now, we consider the family of equations ($s\in [0,1]$)
\begin{align}
\begin{split}
\label{equation6}
H(\nabla^2 v,\nabla v, v)&=\Xi^s(\nabla v,v,x)\quad\textrm{in }\, \Omega\\ v&=\ln\phi \quad\textrm{on }\, \partial\Omega.
\end{split}
\end{align}
From Proposition \ref{proposition}, the Evans-Krylov Theorem and by the standard regularity theory for
 second order uniformly elliptic equations we get the higher order estimate 
\begin{align}
\label{c4forvt}
\| v^s\|_{C^{4,\alpha} (\bar\Omega)}<  C\quad\textrm {independent of } s,
\end{align}
for any solution $v^s$ of equation \eqref{equation6} satisfying $v^s\geq \underline v.$ We also point out that, if $s>0$ and
$v^s\geq \underline{v}$ is a solution of \eqref{equation6} then $v^s$ is a supersolution of \eqref{equation5} for
$t=s$. Thus, we have the strictly inequality $v^s>\underline v$ for $s>0$.

Let $C_0^{4,\alpha}(\bar\Omega)$ be the subspace of $C^{4,\alpha}(\bar\Omega)$ consisting of functions vanishing on
the boundary. Consider the cone
\begin{align*}
\mathcal{O}  = \{ & z\in C_0^{4,\alpha}(\bar\Omega)  :  z>0 \textrm{ in } \Omega, \, 
 \nabla_n z >0   \textrm{ on }  \partial\Omega, \\ & z+\underline{v} \textrm{ is admissible} \textrm{ and }
 \|z\|_{C^{4,\alpha}(\bar\Omega)}\leq  C+\|\underline{v}\|_{C^{4,\alpha}(\bar\Omega)}\}, 
\end{align*}
where $C$ is the constant given in \eqref{c4forvt}. Now, we construct a map from $\mathcal{O}\times [0,1]$ to
$C^{2,\alpha}(\bar\Omega)$ given by
\begin{equation*}
M_s[z]=H(\nabla^2(z+\underline{v}), \nabla (z+\underline{v}), z+\underline{v})-
\Xi^s(\nabla(z+\underline{v}), z+\underline{v}, x),\quad z\in\mathcal{O},
\end{equation*}
where $\Xi^s$ is the function given in \eqref{phi-t} and hence
\begin{equation*}
\Xi^s(\nabla(z+\underline{v}), z+\underline{v}, x)= \big(s+(1-s)e^{2z}\big)\Psi(\nabla(z+\underline{v}), z+\underline{v}, x).
\end{equation*} 
Clearly, $z$ is a solution of $M_s[z]=0$ iff $v^s=z+\underline{v}$ is a solution of \eqref{equation6}.
In particular,  $z^0=v^0-\underline{v}$ is the unique solution of $M_0[z]=0$ and $z^0\in\mathcal{O}.$
Moreover, there is no solution of $M_s[z]=0$ on $\partial\mathcal{O}$ for any $s>0$. Therefore, the degree of $M_s$ on 
$\mathcal{O}$ at $0$, $\deg(M_s,\mathcal{O},0)$, is well defined and independent of $s.$ For more details, we refer the
reader to \cite{LI1} and \cite{LI2}.

Now we compute  $\deg(M_0,\mathcal{O},0).$ As $\Theta^1=\Xi^0$, we know that  $M_0[z]=0$ has a unique solution $z^0$ in $\mathcal{O}.$ The Fr\'echet derivative of $M_0$ at $z^0$ is a linear elliptic operator from $C^{4,\alpha}_0(\bar\Omega)$ to
$C^{2,\alpha}(\bar\Omega)$ given by
\begin{align*}
M_{0,z^0}(h)=H^{ij}|_{v^0}\nabla_{ij}h+H^i|_{v^0}\nabla_i h+(H_{v}|_{v^0}-\Xi^0_{v}|_{v^0})h
\end{align*}
where $v^0=z^0+\underline v$. Since $\big(H_{v}-\Xi^0_v\big)|_{v^0}\leq 0$, we have that
$M_{0,z^0}$ is invertible. By the theory in \cite{LI1}, we can see that
\begin{align*}
\deg(M_0,\mathcal{O},0)=\deg(M_{0, z^0},B_1,0)=\pm 1\neq 0,
\end{align*}
where $B_1$ is the unit ball of $C_0^{4,\alpha}(\bar\Omega).$ Therefore,
\begin{align*}
\deg(M_s,\mathcal{O},0)\neq 0 \quad \textrm{for all}\,  s\in [0,1].
\end{align*}
Then, equation $M_s[z]=0$ has at least one solution for any $s\in[0,1].$ In particular,
the function $v^1=z^1+\underline{v}$ is then a solution of \eqref{equation6}. Therefore
$u= e^{v^1}$ is a solution of \eqref{equation}.


\end{document}